\documentclass[11pt]{article}
\usepackage{amssymb}
\usepackage{amsmath,amsthm,enumerate,verbatim}
\usepackage{amsfonts}
\usepackage[pdftex,pdfstartview=FitH]{hyperref}
\usepackage{color}
\usepackage{marginnote}
\hypersetup{colorlinks=false,linkbordercolor={1 1 1},citebordercolor={1 1 1}}
\newtheorem {theorem}{Theorem}[section]
\newtheorem {proposition}{Proposition}[section]

\newtheorem {lemma}{Lemma}[section]

\newtheorem {example}{Example}[section]
\newtheorem {definition}{Definition}[section]
\newtheorem {remark}{Remark}[section]
\def\R{{\mathbb{R}}}

\title{\bf \sf SOS-convex Semi-algebraic Programs and its Applications to Robust Optimization: A Tractable Class of Nonsmooth Convex Optimization \thanks{Research was partially supported by a research grant from Australian Research Council.}
}

\author{N. H. Chieu\thanks{School of Mathematics and Statistics, University of New South Wales, Sydney NSW 2052, Australia, and  Department of Mathematics, Vinh University, Vinh, Nghe An 42118, Vietnam. Email: nhchieu@unsw.edu.au,  nghuychieu@vinuni.edu.com.}, J.W. Feng \thanks{School of Civil and Environmental Engineering, University of New South Wales, Sydney NSW 2052, Australia.
E-mail: jinwen.feng@unsw.edu.au},  W. Gao \thanks{School of Civil and Environmental Engineering, University of New South Wales, Sydney NSW 2052, Australia.
E-mail: w.gao@unsw.edu.au}, G. Li\thanks{School of Mathematics and Statistics, University of New South Wales, Sydney NSW 2052, Australia.
E-mail: g.li@unsw.edu.au} \ and \ D. Wu \thanks{School of Civil and Environmental Engineering, University of New South Wales, Sydney, NSW 2052, Australia. Email: di.wu@unsw.edu.au}}

\date{{\bf Dedicated to the memory of Jon Borwein who was of great inspiration to us} \\
\bigskip
 \today }

\begin{document}
\maketitle

\begin{abstract}
In this paper, we introduce a new class of nonsmooth convex functions called SOS-convex semialgebraic functions extending the recently proposed  notion of SOS-convex polynomials. This class of nonsmooth convex functions covers many common nonsmooth functions arising in the applications
such as the Euclidean norm, the maximum eigenvalue function and
the least squares functions with $\ell_1$-regularization or elastic net regularization used in statistics and compressed sensing. We show that, under commonly used strict feasibility conditions,
the optimal value and an optimal solution of SOS-convex semi-algebraic programs can be  found
by solving a single semi-definite programming problem (SDP). We achieve the results by using
tools from semi-algebraic geometry, convex-concave minimax theorem and a recently established Jensen inequality type result for SOS-convex polynomials.
As an application, we outline how the derived results can be applied to show that robust SOS-convex optimization problems under restricted spectrahedron data uncertainty enjoy exact SDP relaxations. This extends the existing exact SDP relaxation result for restricted ellipsoidal data uncertainty and answers the open questions left in \cite{JLV15} on how to recover a robust solution from the semi-definite programming relaxation in this broader setting.
\end{abstract}

{\bf Keywords:} Nonsmooth optimization, Convex optimization, SOS-convex polynomial, semi-definite program, robust optimization.
\section{Introduction}

Convex optimization is ubiquitous across science and engineering \cite{BN01, Jon}. It has found applications in a wide range of disciplines,
 such as automatic control systems, signal processing,  electronic circuit design, data analysis, statistics (optimal design), and finance (see \cite{BN01, Boyd} and the references therein). 
 The key to the success in solving convex optimization problems is that convex functions exhibit a local to global phenomenon: every local minimizer is a global minimizer.
Despite the great success of theoretical and algorithmic development and its wide application, we note that a convex optimization problem is, in general,
NP-hard  from the complexity point of view.

Recently, for convex polynomials, a new notion of sums-of-squares-convexity (SOS-convexity) \cite{Parrilo,HN10} has been proposed as a tractable sufficient condition for convexity
based on semidefinite programming. The SOS-convex polynomials cover many commonly used convex polynomials such as convex quadratic functions and convex separable polynomials. An appealing feature of an SOS-convex polynomial is that deciding whether a polynomial is SOS-convex or not can be equivalently rewritten
as a feasibility problem of a semi-definite programming problem (SDP) which can be validated efficiently. It has also been recently shown that for an SOS-convex optimization problems, its optimal value and optimal solution
can be found by solving a single semi-definite programming problem \cite{Lasserre1} (see also \cite{jl1,jl_ORL}).
%\marginnote{\begin{color}{blue}\small{The red sentence should be supported\\  by citations}\end{color}}
 On the other hand, many modern applications of optimization to the area of statistics, machine learning, signal processing and image processing often result in
structured nonsmooth convex optimization problems \cite{Boyd}.  These optimization problems often take the following generic form
$\min_{x \in \mathbb{R}^{n}} \{g(x)+h(x)\}$,
where $g:\mathbb{R}^{n} \rightarrow \mathbb{R}$ is a convex quadratic function and $h:\mathbb{R}^{n} \rightarrow \mathbb{R}$ is a nonsmooth function.
For  example, in many signal processing applications $g$ represents the quality of the recovered signal while $h$ serves as a regularization
which enforces prior knowledge of the form of the signal, such as simplicity/sparsity  (in the sense that the solution has fewest nonzero entries).  Some typical choices of the regularization function promoting the sparsity of the solution are the
so-called $\ell_1$-norm and the weighted sum of $\ell_1$-norm and $\ell_2$-norm (referred as the elastic net regularization \cite{Zou}),
and is therefore nonsmooth. With these applications in mind, this then motivates the following natural and important question:
\begin{center}
{\it Is it possible to extend the SOS-convex polynomials and SOS-convex optimization problems to the nonsmooth setting which not only covers broad nonsmooth problems
arising in common applications but also maintains the appealing feature of tractability (in terms of semidefinite programming)?
}
\end{center}

The purpose of this paper is to provide an affirmative answer for the above question. In particular, in this paper, we make the following contributions:
\begin{itemize}
 \item[{\rm (1)}] In Section 3, we identify a new class of nonsmooth convex functions which we refer as SOS-convex semi-algebraic functions (Definition 3.1).
This class of nonsmooth convex functions covers not only convex functions which can be expressed as the  maximum of
finitely many SOS-convex polynomials (in particular, SOS-convex polynomials) but also many common nonsmooth functions arising in the applications
such as the Euclidean norm, the maximum eigenvalue function (by identifying
the symmetric matrix spaces $S^n$ as an Euclidean space with dimension $n(n+1)/2$) and
the least squares functions with $\ell_1$-regularizer or elastic net regularizer used in compressed sensing.

\item[{\rm (2)}] In Section 4, we show that, under a commonly used strict feasibility condition,
the optimal value and an optimal solution of SOS-convex semi-algebraic optimization problems can be  found
by solving a single semi-definite programming problem which extends the previous known result of SOS-convex polynomial optimization problems (Theorem \ref{th:exact_SDP} and Theorem \ref{th:3.2}). We achieve this by exploiting
tools from semi-algebraic geometry, convex-concave minimax theorem and a recently established Jensen inequality type result for SOS-convex polynomials.

\item[{\rm (3)}] In Section 5,  we briefly outline how our results can be applied to show that robust SOS-convex optimization problems under restricted spectrahedron data uncertainty enjoy exact
semi-definite programming relaxations. This  extends the existing result for restricted ellipsoidal data uncertainty established in \cite{jl1} and
answers the open questions left in \cite{jl1} on how to recover a robust solution from the semi-definite programming relaxation in this broader setting.
\end{itemize}

\section{Preliminaries}
First of all, let us recall some  notations and basic facts on  sums-of-squares polynomial and semi-definite programming problems.  Recall that $S_{n}$ denotes the space of symmetric $(n\times n)$ matrices
with the
trace inner product and $\succeq $ denotes the L\"{o}wner partial order of $
S^{n}$, that is, for $M,N\in S^{n},$ $M\succeq N$ if and only if
$(M-N)$ is positive semidefinite. Let $S^{n}_+:=\{M\in S_{n}\mid M\succeq 0\}$ be the closed convex cone of positive semidefinite symmetric $(n\times n)$ matrices.
Note that for $M,N\in S^{n}_+$, the inner product,
$(M,N):=\mathrm{Tr\ }[MN]$, where $\mathrm{Tr\ }[.]$ refers to the
trace operation. Note also that
$%
 M\succ 0$ means that $M$ is positive definite. In the sequel, unless otherwise stated, the space $\R^n$ is equipped with  the Euclidean norm, that is, $\|x\|:=(\sum\limits_{i=1}^n |x_i|^2)^{1/2}$ for all $x=(x_1,x_2,...,x_n)\in \R^n$.  Consider a polynomial $f$ with degree at most $d$ where $d$ is an even number. Let
$\mathbb{R}_d[x_1,\ldots,x_n]$ be the space consisting of all real
polynomials on $\mathbb{R}^n$ with degree at most $d$   and let $s(d,n)$ be the dimension of
$\mathbb{R}_d[x_1,\ldots,x_n]$.
Write the canonical basis of
$\mathbb{R}_d[x_1,\ldots,x_n]$ by
\[
x^{(d)}:=(1,x_1,x_2,\ldots,x_n,x_1^2,x_1x_2,\ldots,x_2^2,\ldots,x_n^2,\ldots,x_1^{d},\ldots,x_n^d)^T
\]
and let $x^{(d)}_{\alpha}$ be the $\alpha$-th coordinate of $x^{(d)}$,
$1 \le \alpha \le s(d,n)$.  Then, we can write
$f(x)= \sum_{\alpha=1}^{s(d,n)} f_{\alpha}x^{(d)}_{\alpha}$. %and $g_i^{(j)}=\sum_{\alpha=1}^{C(k_0,n)} \big(g_i^{(j)}\big)_{\alpha}x^{(d)}_{\alpha}$.

 We say that a real polynomial $f$ is sums-of-squares (cf. \cite{Lasserre}) if there exist real polynomials $f_j$, $j=1,\ldots,r$, such that $f=\sum_{j=1}^rf_j^2$. The set consisting of all sum of squares real polynomials in the variable $x$ is denoted by $\Sigma^2[x]$. Moreover, the set consisting of all sum of squares real polynomials with degree at most $d$ is denoted by $\Sigma^2_d[x]$.  For a polynomial $f$, we use ${\rm deg}f$ to denote the degree of $f$. Let $l=d/2$. Then, $f$ is a sum-of-squares polynomial if and only if there exists a  positive semi-definite symmetric matrix $W \in S_+^{s(l,n)}$ such that
\begin{equation}\label{eq:useful}
%\bigg(f_1+\sum\limits_{i=1}^{m}\sum_{j=0}^{s_i} w_i^{(j)} \beta_i^{(j)}-\mu\bigg)+ \bigg(v+\sum\limits_{i=1}^{m}\sum_{j=0}^{s_i} w_i^{(j)}\bigg)^Tx+  \sum_{\alpha=n+2}^{C(d,n)} f_{\alpha}x^{(d)}_{\alpha}=
f(x)=(x^{(l)})^TW x^{(l)},
\end{equation}
where $x^{(l)}=(1,x_1,x_2,\ldots,x_n,x_1^2,x_1x_2,\ldots,x_2^2,\ldots,x_n^2,\ldots,x_1^{l},\ldots,x_n^l)^T$. For each $1 \le \alpha \le s(d,n)$, we denote $i(\alpha)=(i_1(\alpha),\ldots,i_n(\alpha)) \in (\mathbb{N} \cup \{0\})^n$ to be the multi-index such that
 $$x^{(d)}_{\alpha}=x^{i(\alpha)}:=x_1^{i_1(\alpha)}\ldots x_n^{i_n(\alpha)}.$$
 Then, by comparing the coefficients in (\ref{eq:useful}), we have the following linear matrix inequality characterization of a sum-of-squares polynomial.
\begin{lemma}\label{th:sos} Let $d$ be an even number.
For a polynomial $f$ on $\mathbb{R}^n$ with degree at most $d$, $f$ is a sum-of-squares polynomial if and only if the following linear matrix inequality problem has a solution
 \begin{eqnarray*}
\left\{ \begin{array}{l}
   %\bigg(v+\sum\limits_{i=1}^{m}\sum_{j=0}^{s_i} w_i^{(j)}\bigg)_{\alpha-1}=W_{1,\alpha}=W_{\alpha,1}, \ \ \ 2 \le \alpha \le n+1 \\
W \in S_+^{s(l,n)} \\
 \displaystyle f_{\alpha}=\sum_{1 \le \beta,\gamma \le s(l,n), i(\beta)+i(\gamma)=i(\alpha)} W_{\beta,\gamma}, \ 1  \le \alpha \le s(d,n),\ l=d/2.
        \end{array}\right.
\end{eqnarray*}
\end{lemma}

 We now recall  the definition of SOS-convex polynomial. The notion of SOS-convex polynomial was first proposed in \cite{HN10} and further developed in \cite{Parrilo}. Here, for convenience of our discussion, we follow the definition used in \cite{Parrilo}.

\begin{definition}[{\bf SOS-Convex Polynomials} {\cite{HN10}}]
A real polynomial $f$ on $\mathbb{R}^n$ is called \textit{SOS-convex} if
the polynomial $F: (x,y) \mapsto f(x)-f(y)-\nabla f(y)^T(x-y)$ is a sums-of-squares
polynomial on $\mathbb{R}^n \times \mathbb{R}^n$.
\end{definition}

The significance of the class of SOS-convex polynomials is that checking whether a polynomial is SOS-convex is equivalent to solving a semi-definite programming problem (SDP)  which can be done in
polynomial time; while checking a polynomial is convex or not is, in general, an NP-hard problem \cite{HN10,Parrilo}. Moreover, another important fact is that, for SOS-convex polynomial program, an exact SDP relaxation
holds under the usual strict feasibility condition. In contrast, solving a convex polynomial program, is again, in general, an NP hard problem \cite{Parrilo}.

Clearly,
a SOS-convex polynomial is convex. However, the converse is not true, that is, there exists a convex polynomial which is not SOS-convex \cite{Parrilo}. The sum of two SOS-convex polynomials and nonnegative scalar multiplication of an SOS-convex polynomial are still SOS convex polynomials. 
It is known that any convex quadratic function and any convex separable polynomial is an SOS-convex polynomial \cite{jl1}. Moreover, an SOS-convex polynomial can be non-quadratic and non-separable. For instance, $f(x)=x_1^8+x_1^2+x_1x_2+x_2^2$ is a SOS-convex polynomial  which is  non-quadratic and non-separable.

The following existence result for  solutions of a convex polynomial optimization problem will also be useful for our later analysis.
\begin{lemma}[{\bf Solution Existence of Convex Polynomial Programs} {\cite[Theorem 3]{convex_polynomial}}]
\label{minattain}
Let $f_0,f_1,\ldots,f_m$ be convex polynomials on $\mathbb{R}^n$ and let $C:=\left\{x \in \mathbb{R}^n : f_i(x) \leq 0, i=1,\ldots,m\right\}$ be nonempty. If $\inf\limits_{x\in C}f_0(x)>-\infty$  then  $\operatorname{argmin}\limits_{x\in C}f_0(x) \neq \emptyset$.
\end{lemma}

\section{SOS-convex semi-algebraic functions}

We begin this section with introducing the notion of SOS-convex semi-algebraic functions. The class of SOS-convex semi-algebraic functions is  a subclass of the class of locally Lipschitz nonsmooth convex functions, and  includes  SOS-convex polynomials.

\begin{definition}{\bf (SOS-convex semi-algebraic functions) }
We say $f:\R^n\rightarrow \R$  is an SOS-convex semi-algebraic function on $\mathbb{R}^n$ if it admits a representation
\begin{equation}\label{rep} \displaystyle f(x)=\sup_{y \in \Omega} \{h_0(x)+\sum_{j=1}^m y_j h_j(x)\},\, m \in \mathbb{N},\end{equation} where
\begin{itemize}
\item[{\rm (1)}] each $h_j$, $j=0,1,\ldots,m$, is a polynomial and for each $y \in \Omega$, $\displaystyle h_0+\sum_{j=1}^m y_j h_j$ is a SOS-convex polynomial on $\mathbb{R}^n$;
\item[{\rm (2)}] $\Omega$ is a nonempty compact semi-definite program representable set given by
\begin{equation}\label{eq:1}
\Omega=\{y \in \mathbb{R}^m: \exists \, z  \in \mathbb{R}^p \mbox{ s.t. } A_0+\sum_{j=1}^m y_j A_j+ \sum_{l=1}^p z_l B_l \succeq 0\},
\end{equation}
\end{itemize}
for some  $p \in \mathbb{N},$ $A_j$ and $B_l,$ $j=0,1,...,m,$ $l=1,...,p,$ being $(t\times t)$-symmetric matrices with some  $t\in \mathbb N.$

 \noindent Moreover, the maximum of the degree of the polynomial $h_j$, $j=1,\ldots,m,$ is said to be  the degree of the  SOS-convex semi-algebraic function $f$ with respect to  the representation~\eqref{rep}.
\end{definition}

The class of SOS-convex semi-algebraic functions contains many common nonsmooth convex functions. Below, we provide some typical examples.
\begin{example}{\bf (Examples of SOS-convex semi-algebraic functions)}\label{ex3.1}
\begin{itemize}
\item[{\rm (1)}] Let $f(x)=\max_{1 \le i \le m}f_i(x)$ where each $f_i$, $i=1,\ldots,m$, is an SOS-convex polynomial. Note that
$f(x)=\sup_{y \in \Delta} g(x,y)$ where $\Delta$ is the simplex in $\mathbb{R}^m$ given by $\Delta=\{y: y_i \ge 0, \sum_{i=1}^m y_i=1\}$
 and $g(x,y)=\sum_{i=1}^m y_i f_i(x)$. Then, we see that $f$ is an SOS-convex semi-algebraic function.
\item[{\rm (2)}] Let $f(x)=\|x\|$.  Then, $f$ is an SOS-convex semi-algebraic function. To see this, we only need to note that
\[
\|x\|=\sup_{\|(y_1,\ldots,y_n)\| \le 1} \sum_{i=1}^n x_iy_i,
\]
and the unit ball defined by $\|\cdot\|$ is a compact semi-definite program representable set. 
More generally, $f(x)=\|x\|_p:=\big(\sum_{i=1}^n |x_i|^p\big)^{\frac{1}{p}}$ with $p=\frac{s}{s-1}$ and $s$ being an even positive integer,  is an SOS-convex semi-algebraic function. 
To see this, we only need to  note that  $$\|x\|_p=\sup_{\|(y_1,\ldots,y_n)\|_s \le 1} \sum_{i=1}^n x_iy_i.$$ and the set $\{y \in \mathbb{R}^n: \|y\|_s \le 1\}=\{y \in \mathbb{R}^n: 
\sum_{i=1}^n y_i^s \le 1\}$ is described by an SOS-convex polynomial inequality (as $s$ is even) and so, is a  compact semi-definite program representable set \cite{HN10}. 
\item[{\rm (3)}] Identify the $(n \times n)$ symmetric matrices space $S^n$ with the trace inner product ${\rm Tr}(AB)=\sum_{ij} A_{ij}B_{ij}$ as $\mathbb{R}^{n(n+1)/2}$ with the usual inner product. Let $f:S^n \rightarrow \mathbb{R}$ be defined by $f(X)=\lambda_{\max}(X)$ where $\lambda_{\max}$ is the maximum eigenvalue. Then, $f$ is an SOS-convex semi-algebraic function on $S^n$. To see this, we only need to notice that
    \[
    \lambda_{\max}(X)=\sup\{{\rm Tr}(XY): Y \in S^n, {\rm Tr}(Y)=1, Y \succeq 0\}
    \]
    and the set $\{Y \in S^n: {\rm Tr}(Y)=1, Y \succeq 0\}$ is a compact semi-definite program representable set.
\end{itemize}
\end{example}

Next, we see that SOS-convex semi-algebraic functions cover many least squares functions with regularization. To see this, we need the following simple lemma which shows that finite addition preserves SOS-convex semi-algebracity.
\begin{proposition}\label{prop3.1}
Let $f_i$ be SOS-convex semi-algebraic functions on $\mathbb{R}^n$, $i=1,\ldots,q$. Then, $\sum_{i=1}^q f_i$ is an
SOS-convex semi-algebraic function on $\mathbb{R}^n$. %Moreover, let $f$ be an SOS-convex semi-algebraic function on $\mathbb{R}^n$ and $A$ be a linear map from $\mathbb{R}^q$ to $\mathbb{R}^n$. Then, $x \mapsto f(Ax)$ is an SOS-convex semi-algebraic function on $\mathbb{R}^q$.
\end{proposition}
\begin{proof}
To see the conclusion, it suffices to show the case where $q=2$. We first show that $f_1+f_2$ is an SOS-convex semi-algebraic function.
As $f_i$, $i=1,2$ are SOS-convex semi-algebraic functions, $f_i(x)=\sup_{y^i \in \Omega_i}\{h_0^i(x)+\sum_{j=1}^{m_i} y_j^i h_j^i(x)\}$,  where $m_i \in \mathbb{N}$, $h_l^i$, $l=0,1,\ldots,m$ are SOS-convex polynomials and
$\Omega_i$ is a compact semi-definite program representable sets given by
\[
\Omega_i=\{y^i \in \mathbb{R}^{m_i}: \exists \, z^i \in \mathbb{R}^{p_i} \mbox{ s.t. } A_0^i+\sum_{j=1}^{m_i} y_j^i A_j^i+ \sum_{l=1}^{p_i} z_l^i B_l^i \succeq 0\}.
\]
%$f_2(x)=\sup_{w \in \Omega_2}\{\phi_0(x)+\sum_{j=1}^m w_j \phi_j(x)\}$,  where $\phi_l$, $l=0,1,\ldots,$
%$\Omega_2$ is a compact semi-definite program representable sets given by
%\[
%\Omega_2=\{w \in \mathbb{R}^{m_2}: \exists \, v \in \mathbb{R}^{p_2} \mbox{ s.t. } C_0+\sum_{j=1}^m w_j C_j+ \sum_{l=1}^{p_2} v_l D_l \succeq 0\}.
%\]
Then,
\[
f_1(x)+f_2(x)=\sup_{(y^1,y^2) \in \Omega_1\times \Omega_2}\left\{h_0^1(x)+h_0^2(x)+\sum_{j=1}^{m_1} y_j^1 h_j^1(x)+\sum_{j=1}^{m_2} y_j^2 h_j^2(x)\right\}.
\]
Note that $\Omega_1 \times \Omega_2$ is also a compact semi-definite program representable set. Thus, $f_1+f_2$ is also an SOS-convex semi-algebraic function.
%
%The second assertion follows directly from the definition.
\end{proof}

\begin{example}{\bf (Further examples: least squares problems with regularization)}
Let $A \in \mathbb{R}^{m \times n}$ and $b \in \mathbb{R}^m$.  From the preceding proposition, we see that the following functions  which arises in sparse optimization are SOS-convex semi-algebraic:
\begin{itemize}
\item[{\rm (1)}] The least squares function with $\ell_1$-regularization $f(x)=\|Ax-b\|^2+ \mu \|x\|_1$ where $\mu >0$. Note that since $\|x\|_1:=|x_1|+|x_2|+...+|x_n|$ and $|x_i|=\max\{x_i,-x_i\},$ $i=1,2,...,n,$ it follows from Example~\ref{ex3.1} and Proposition~\ref{prop3.1} that  $\|\cdot\|_1$ is an SOS-convex semi-algebraic function, while the function $x\mapsto\|Ax-b\|^2$ is a convex quadratic function and thus is SOS-convex semi-algebraic.
\item[{\rm (2)}] The least squares function with elastic net regularization \cite{Zou} $f(x)=\|Ax-b\|^2+ \mu_1 \|x\|_1+ \mu_2 \|x\|^2$ where $\mu_1,\mu_2>0$.
\end{itemize}
\end{example}

\section{Exact SDP relaxation for SOS-convex semi-algebraic programs}

In this section, we show that an SOS-convex semi-algebraic program admits an exact SDP relaxation in the sense that the optimal value of the SDP relaxation problem equals the optimal value of the underlying SOS-convex semi-algebraic program. Moreover, a solution for the SOS-convex semi-algebraic program can be recovered from its SDP relaxation, under strict feasibility assumptions.

Consider the following SOS-convex semi-algebraic program:
\begin{eqnarray*}
(P) & \min & f_0(x) \\
& \mbox{ subject to } & f_i(x) \le 0, \ i=1,\ldots,s,
\end{eqnarray*}
where each $f_i$, $i=0,1,\ldots,s$, is an SOS-convex semi-algebraic function  in the form
$$\displaystyle f_i(x)=\sup_{(y_1^i,\ldots,y_m^i) \in \Omega_i} \{h_0^i(x)+\sum_{j=1}^m y_j^i h_j^i(x)\}, \ m \in \mathbb{N},$$ such that
\begin{itemize}
%\item[{\rm (1)}]  $h$ is a polynomial on $\mathbb{R}^n \times \mathbb{R}^m$;
\item[{\rm (1)}] each $h_j^i$ is a polynomial with degree at most $d$,  and for each $y^i=(y^i_1,...,y^i_m)\in \Omega_i$, the function $\displaystyle h_0^i+\sum_{j=1}^m y^i_j h_j^i$ is an SOS-convex polynomial on $\mathbb{R}^n$;
%\item[{\rm (2)}] for each $x \in \mathbb{R}^n$, $h(x,\cdot)$ is a concave function;
\item[{\rm (2)}] $\Omega_i$, $i=0,1,\ldots,s$, is a nonempty compact semi-definite program representable set given by
\begin{equation*}
\Omega_i=\big\{(y_1^i,\ldots,y_m^i) \in \mathbb{R}^m: \exists  z^i=(z^i_1,...,z^i_{p_i})  \in \mathbb{R}^{p_i} \mbox{ s.t. } A_0^i+\sum_{j=1}^m y_j^i A_j^i+ \sum_{l=1}^{p_i} z_l^i B_l^i \succeq 0\big\},
\end{equation*}
for some $p_i \in \mathbb{N}.$
\end{itemize}
{\em Without loss of generality, throughout this paper, we assume that $d$ is an even number. }

We now introduce  a  relaxation problem for problem (P) as follows
\begin{eqnarray*}
({SDP}) \ \ \ \sup_{\substack{\lambda_0^i \ge 0, (\lambda_1^i,...,\lambda_m^i) \in \mathbb{R}^m \\
 z_l^i \in \mathbb{R}, \mu \in \mathbb{R}}} \big\{\mu&:& h_0^0+\sum_{j=1}^m \lambda_j^0  h_j^0+\sum_{i=1}^s \left(\lambda_0^i h_0^i + \sum_{j=1}^m \lambda_j^i h_j^i\right) -\mu \in \Sigma^2_d[x], \\
& & \ \ \ \ \ \ \ \ \ \ \ \ \ \ \ \ \ \ \ \ \  A_0^0+\sum_{j=1}^m\lambda_j^0 A_j^0+\sum_{l=1}^{p_0} z_l^0 B_l^0 \succeq 0,\\
& & \ \ \ \ \ \ \ \ \ \ \ \ \ \ \ \ \ \ \ \ \  \lambda_0^i A_0^i+\sum_{j=1}^m\lambda_j^i A_j^i+\sum_{l=1}^{p_i} z_l^i B_l^i \succeq 0, \, i=1,\ldots,s\big\}. \end{eqnarray*}
We note (from Lemma \ref{th:sos}) that $(SDP)$ can be equivalently rewritten as the following semi-definite programming problem:
{\small \begin{eqnarray*}
& & \sup_{\substack{\lambda_0^i \ge 0,   (\lambda_1^i,...,\lambda_m^i) \in \mathbb{R}^m\\
 z_l^i \in \mathbb{R}, \mu \in \mathbb{R},
W \in S^{s(d/2,n)}}} \{\mu:   (h_0^0)_1+\sum_{j=1}^m  \lambda_j^0 (h_j^0)_1+\sum\limits_{i=1}^{s}\left(\lambda_0^i (h_0^i)_1 + \sum_{j=1}^m \lambda_j^i (h_j^i)_1\right)-\mu =W_{1,1}, \nonumber \\
& & \ \ \ \ \ \ \ \ \ \ \ \ \ \ \ \ \ \ \ \ \  (h_0^0)_\alpha+\sum_{j=1}^m \lambda_j^0 (h_j^0)_\alpha+\sum\limits_{i=1}^{s}\left(\lambda_0^i (h_0^i)_\alpha + \sum_{j=1}^m \lambda_j^i (h_j^i)_\alpha\right)=\sum_{\substack{ 1\le \beta,\gamma \le s(d/2,n) \\i(\beta)+i(\gamma)=i(\alpha)}}W_{\beta,\gamma} \, ,\ 2  \le \alpha \le s(d,n)\nonumber \\
%f+\sum_{i=1}^m \left(\lambda_i^0 g_i^0 + \sum_{j=1}^s \lambda_i^j g_i^j\right) -\mu \in \Sigma^2_d, \\
& & \ \ \ \ \ \ \ \ \ \ \ \ \ \ \ \ \ \ \ \ \  W \succeq 0,  \quad  \,  A_0^0+\sum_{j=1}^m\lambda_j^0 A_j^0+\sum_{l=1}^{p_0} z_l^0 B_l^0 \succeq 0,\\
& & \ \ \ \ \ \ \ \ \ \ \ \ \ \ \ \ \ \ \ \ \ \lambda_0^i A_0^i+\sum_{j=1}^m\lambda_j^i A_j^i+\sum_{l=1}^{p_i} z_l^i B_l^i \succeq 0, \, i= 1,\ldots,s\}. \nonumber
\end{eqnarray*}}

Next, we show that an exact SDP relaxation holds between (P) and $(SDP)$ in the sense that their optimal values are the same.
We start with a simple property for a bounded set which describes by linear matrix inequalities.
\begin{lemma} \label{lemma:1}
Let $\mathcal{U}$ be a nonempty compact set with the form $\mathcal{U}=\{(u_1,\ldots,u_m) \in \mathbb{R}^m: \exists z \in \mathbb{R}^p \mbox{
 such that } A_0+\sum_{j=1}^mu_j A_j+ \sum_{l=1}^p z_l B_l \succeq 0\}$ where $A_j,B_l \in S^{q}$.
 Let $(\lambda_0,\ldots,\lambda_m) \in \mathbb{R}^{m+1}$  and  $\lambda_0 A_0+\sum_{j=1}^m\lambda_j A_j+ \sum_{l=1}^p v_l B_l \succeq 0$ for some
 $(v_1,\ldots,v_p) \in \mathbb{R}^p$.  Then, the following
 implication holds: \begin{equation}\label{eq:trick}
\lambda_0=0 \ \Rightarrow \ \lambda_j=0 \  \mbox{ for all } j=1,\ldots,m.
                                                         \end{equation}
 %\item[{\rm (3)}] $\mathcal{U}^{\infty}=\{0\}$, where $\mathcal{U}^\infty$ is the recession cone of $\mathcal{U}$.
%\end{itemize}
                                                         \end{lemma}
\begin{proof}
% We now observe that
%  for each $i=1,\ldots,m$, the following implicaton holds \begin{equation}\label{eq:trick}
% \lambda_i^0=0 \ \Rightarrow \ \lambda_i^j=0, \mbox{ for all } j=1,\ldots,s.
%                                                          \end{equation}
%  To see this
  Let $(\lambda_0,\ldots,\lambda_m) \in \mathbb{R}^{m+1}$  and $\lambda_0 A_0+\sum_{j=1}^m\lambda_j A_j+ \sum_{l=1}^p v_l B_l \succeq 0$ for some
 $(v_1,\ldots,v_p) \in \mathbb{R}^p$.  We proceed by the method of contradiction. Suppose that $\lambda_0=0$ and  there exists  $j_0 \in \{1,\ldots,m\}$ with  $\lambda_{j_0} \neq 0.$ This means  that
\[
\sum_{j=1}^m\lambda_j A_j+ \sum_{l=1}^p v_l B_l \succeq 0 \mbox{ and } (\lambda_1,\ldots,\lambda_m)  \neq 0_{\mathbb{R}^m}.
\]
Now take $\hat{u}=(\hat u_1,\ldots,\hat u_m) \in \mathcal{U}$. Then, we have
 $A_0+\sum_{j=1}^m\hat{u}_j A_j+ \sum_{l=1}^p \hat{v}_l B_l \succeq 0$ for some $(\hat{v}_1,\ldots,\hat{v}_p)\in \R^p$,
and so,
\[
 A_0+\sum_{j=1}^m (\hat{u}_j+ t \lambda_j) A_j + \sum_{l=1}^p (\hat{v}_l+t v_l) B_l\succeq 0\  \mbox{ for all } t \ge 0.
\]
The latter implies that
\[
(\hat u_1,\ldots,\hat u_m)+t(\lambda_1,\ldots,\lambda_m) \in \mathcal{U}\  \mbox{ for all } t \ge 0,
\]
 which contradicts the boundedness of $\mathcal{U}$. Thus, the conclusion follows.
%
% [${\rm (2)} \Rightarrow {\rm (3)}$] Let $v=(v^1,\ldots,v^s) \in \mathcal{U}^{\infty}$ and $\bar u=(\bar u^1,\ldots,\bar u^s) \in \mathcal{U}$. Then, for all $t \ge 0$,
% \[
% A_0+ \sum_{j=1}^s \bar{u}^j A_j + t(\sum_{j=1}^s v^j A_i^j) \succeq 0.
% \]
% This forces that $\displaystyle \sum_{j=1}^s v^j A_i^j \succeq 0$. Then, {\rm (2)} implies that $v=0$, and so, $\mathcal{U}^{\infty}=\{0\}$.
%
% [${\rm (3)} \Leftrightarrow {\rm (1)}$] This equivalence is a classical result in convex analysis.
 \end{proof}

\smallskip

% We now show that an exact SDP relaxation result holds for (RP), in the sense that, the optimal value of the (RP) can be found by solving a single semi-definite programming problem
% under the assumptions that ``$f$ is SOS-convex and, for each fixed $u_i \in \mathcal{U}_i$, $g_i^0+\sum_{j=1}^s u_i^j g_i^j$ is SOS-convex''.
%  We note that this assumption is satisfied when $f$ and $g_i^0$ are SOS-convex, and $g_i^j$, $j=1,\ldots,s$ are affine (for other sufficient condition see
%  Proposition \ref{prop:4.1} later). We also remark that
%  an exact SDP relaxation can fail if we only require $f$ and $g_i^j$, $i=1,\ldots,m$, $j=0,1,\ldots,s$ are SOS-convex, because it has been shown in \cite{BEN09} (see also \cite[Remark 1]{Gold})
% that robust convex quadratic optimization probems under ellipsoidal data uncertainty set are, in general, NP-hard.

We are now ready to state and prove the first main result of this section, showing the exactness of the SDP relaxation for SOS-convex semi-algebraic programs under a strict feasibility condition.

\begin{theorem}{\bf (Exact SDP Relaxation for SOS-convex Semi-algebraic Programs)}\label{th:exact_SDP}
For  problem~$(P),$   suppose  the following  strict feasibility condition holds: there  exists $x_0 \in \mathbb{R}^n$ such that $f_i(x_0)<0$, $i=1,\ldots,s$.
Then, we have \begin{eqnarray*} {\rm val}(P)&=& {\rm val}(SDP),\end{eqnarray*}
where ${\rm val}(P)$ and  ${\rm val}(SDP)$ are the optimal values of problems $(P)$ and $(SDP),$ respectively. 
\end{theorem}
\begin{proof} We first justify that ${\rm val}(P) \ge {\rm val}(SDP).$ Let  $\lambda_0^i \ge 0,$ $(\lambda_1^i,...,\lambda_m^i) \in \mathbb{R}^m,$
 $z_l^i \in \mathbb{R},$  and $\mu \in \mathbb{R},$  be feasible for $(SDP)$. Then, we have
\begin{eqnarray*}
& & h_0^0+\sum_{j=1}^m  \lambda_j^0 h_j^0+\sum_{i=1}^s \left(\lambda_0^i h_0^i + \sum_{j=1}^m \lambda_j^i h_j^i\right) -\mu \in \Sigma^2_d[x], \\
 & & A_0^0+\sum_{j=1}^m\lambda_j^0 A_j^0+\sum_{l=1}^{p_0} z_l^0 B_l^0 \succeq 0,\\
& &   \lambda_0^i A_0^i+\sum_{j=1}^m\lambda_j^i A_j^i+\sum_{l=1}^{p_i} z_l^i B_l^i \succeq 0, \, i=1,\ldots,s .
\end{eqnarray*}
Take any  $x \in \mathbb{R}^n$ with $f_i(x) \le 0,$ $i=1,...,s$. We want to show $f_0(x) \ge \mu$.  For each $i=1,\ldots,s,$ pick  $\bar y_i=(\bar y_1^i,\ldots,\bar y_m^i) \in \Omega_i.$ By the definition of $\Omega_i,$   there exist $\bar z^i  \in \mathbb{R}^{p_i}$ such that
$A_0^i+\sum_{j=1}^m \bar y_j^i A_j^i+ \sum_{l=1}^{p_i} \bar z_l^i B_l^i \succeq 0$.
For each $j=1,\ldots,m$ and $i=1,\ldots,s,$ put
\[
\widetilde{y}_j^i:= \left\{\begin{array}{cc}
              \frac{\lambda_j^i}{\lambda_0^i} &\mbox{if}\  \lambda_0^i \neq 0, \\
                         \bar y_j^i                    &\mbox{if}\   \lambda_0^i = 0,
              \end{array}
 \right.
\]
and
\[
\widetilde{z}_l^i:= \left\{\begin{array}{cc}
              \frac{z_l^i}{\lambda_0^i}&\mbox{if}\  \lambda_0^i \neq 0, \\
                         \bar z_l^i &\mbox{if}\   \lambda_0^i = 0.
              \end{array}
 \right.
\]
Then, for each $i=1,\ldots,s$, we have
\[
A_i^0+\sum_{j=1}^m \widetilde{y}_j^i A_j^i+ \sum_{l=1}^{p_i} \tilde{z}_l^i B_l^i  = \left\{\begin{array}{cc}
              \frac{1}{\lambda_0^i} \big(\lambda_0^i A_0^i+\sum_{j=1}^m \lambda_j^i A_j^i+ \sum_{l=1}^{p_i} z_l^i B_l^i\big)& \mbox{if}\ \lambda_0^i \neq 0, \\
                         A_0^i+\sum_{j=1}^m \bar y_j^i A_j^i+ \sum_{l=1}^{p_i} \bar z_l^i B_l^i                     &\mbox{if}\   \lambda_0^i = 0,
              \end{array}
 \right.
\]
which is always a positive semidefinite symmetric matrix. So, $\widetilde{y}^i:=(\widetilde{y}^i_1,...,\widetilde{y}^i_m) \in \Omega_i$ and hence
\begin{equation}\label{tv1}h_0^i(x)+ \sum_{j=1}^m \widetilde{y}_j^i h_j^i(x) \le f_i(x)\leq 0\quad\mbox{for all}\  i=1,...,s.\end{equation}
Moreover, as the sets $\mathcal{U}_i$ are bounded, according to  Lemma \ref{lemma:1},  for each $i=1,\ldots,s$,
if $\lambda_0^i=0$, then $\lambda_j^i=0$ for all $j=1,\ldots,m$. This implies that
\begin{eqnarray*}
h_0^0+\sum_{j=1}^m \lambda^0_j h_j^0+\sum_{i=1}^s \lambda_0^i \left( h_0^i + \sum_{j=1}^m \widetilde{y}_j^i h_j^i\right)-\mu
  &=&  h_0^0+\sum_{j=1}^m \lambda^0_j h_j^0+\sum_{i=1}^s \left( \lambda_0^i  h_0^i + \sum_{j=1}^m (\lambda_0^i \widetilde{y}_j^i)  h_j^i\right)-\mu  \\
 &= &  h_0^0+\sum_{j=1}^m \lambda^0_j h_j^0+\sum_{i=1}^s \left( \lambda_0^i  h_0^i + \sum_{j=1}^m \lambda_j^i  h_j^i\right)-\mu \in \Sigma^2_d[x].
\end{eqnarray*}
So, noting that $(\lambda^0_1,\ldots,\lambda^0_m) \in \Omega_0$ and $\lambda_0^i\geq 0$ for all $i=1,...,s$, by \eqref{tv1}  it holds that
\[\begin{array}{rl}
f_0(x) &\ge h_0^0(x)+\sum\limits_{j=1}^m \lambda^0_jh_j^0(x) \\ \cr
& \ge h_0^0(x)+\sum\limits_{j=1}^m \lambda^0_j  h_j^0(x)+\sum\limits_{i=1}^s \lambda_0^i \left( h_0^i(x) + \sum\limits_{j=1}^m \widetilde{y}_j^i h_j^i(x)\right) \ge \mu.\end{array} \]
Therefore, ${\rm val}(P) \ge {\rm val}(SDP)$.

Next, we will justify that  ${\rm val}(P) \le {\rm val}(SDP).$  As ${\rm val}(P) \ge {\rm val}(SDP)$ always holds, it suffices to consider the case   ${\rm val}(P)>-\infty$. Noting that the feasible set of $(P)$  is nonempty, we may assume that $r:={\rm val}(P) \in \mathbb{R}$.
Our assumptions guarantee that   there exists $x_0$ such that $f_i(x_0)<0$, $i=1,\ldots,s,$ and each $f_i$ is a continuous convex function. So, the standard Lagrangian duality for convex programming problem shows  that
\begin{equation}\label{tv2}\begin{array}{rl}
r &:= \inf\limits_{x \in \mathbb{R}^n}\big\{f_0(x): f_i(x) \le 0,\ i=1,...,s\big\} \\ \cr
& =  \max\limits_{\lambda\in \R^s_+} \inf\limits_{x \in \mathbb{R}^n}\big\{f_0(x)+\sum\limits_{i=1}^s \lambda_i f_i(x)\big\}\\ \cr
&= \max\limits_{\lambda\in \R^s_+} \inf\limits_{x \in \mathbb{R}^n}\max\limits_{y\in  \prod\limits_{i=0}^s \Omega_i}h_\lambda(x,y), \end{array}
\end{equation}
where $\lambda:=(\lambda_1,...,\lambda_s)\in \R^s,$  $y:=(y_1^0,\ldots,y_m^0,...,y_1^s,\ldots,y_m^s)\in \R^{m(s+1)},$ and
$$h_\lambda(x,y):=h_0^0(x)+\sum_{j=1}^m y_j^0 h_j^0(x)+\sum_{i=1}^s \lambda_i\big(h_0^i(x)+\sum_{j=1}^m y_j^i h_j^i(x)\big).$$
Note that   $ \prod\limits_{i=0}^s \Omega_i$ is a convex compact set,  and for any  $\lambda\in \R^s_+$ the function $h_\lambda(x,y)$ is convex in $x$ for each fixed~$y$ and is concave in $y$ for each fixed $x.$ Thus, for each $\lambda\in \R^s_+,$ by the convex-concave minimax theorem we have
\begin{eqnarray*}
 \inf_{x \in \mathbb{R}^n}\max_{y\in  \prod\limits_{i=0}^s \Omega_i}h_\lambda(x,y)= \max_{y\in  \prod\limits_{i=0}^s \Omega_i}\inf_{x \in \mathbb{R}^n}h_\lambda(x,y).
\end{eqnarray*}
This together with \eqref{tv2} yields
$$\begin{array}{rl}
r  &= \max\limits_{\lambda\in \R^s_+} \max\limits_{y\in  \prod\limits_{i=0}^s \Omega_i}\inf\limits_{x \in \mathbb{R}^n}h_\lambda(x,y)\\ \cr
&=  \max\limits_{\substack{(y_1^i,\ldots,y_m^i) \in \Omega_i, 0 \le i\le s \\
\lambda_1 \ge 0,...,\lambda_s\geq0}}\inf\limits_{x \in \mathbb{R}^n}\big\{h_0^0(x)+\sum\limits_{j=1}^m y_j^0 h_j^0(x)+\sum\limits_{i=1}^s \lambda_i\big(h_0^i(x)+\sum\limits_{j=1}^m y_j^i h_j^i(x)\big)\big\}.\end{array}
$$
In particular, the latter  shows that there exist $(\tilde{y}_1^i,\ldots,\tilde{y}_m^i) \in \Omega_i,$ $0 \le i\le s,$ and $\tilde{\lambda}_1 \ge 0,...,\tilde{\lambda}_s \ge 0,$  such that
\[
\inf_{x \in \mathbb{R}^n}\big\{h_0^0(x)+\sum_{j=1}^m \tilde{y}_j^0 h_j^0(x)+\sum_{i=1}^s \tilde{\lambda}_i\big(h_0^i(x)+\sum_{j=1}^m \tilde{y}_j^i h_j^i(x)\big)\big\} =r.
\]
Denote $G(x)=h_0^0(x)+\sum_{j=1}^m \tilde{y}_j^0 h_j^0(x)+\sum_{i=1}^s \tilde{\lambda}_i\big(h_0^i(x)+\sum_{j=1}^m \tilde{y}_j^i h_j^i(x)\big)-r$.
By  Lemma~\ref{minattain}, there exists $a \in \mathbb{R}^n$ such that $G(a)=\inf_{x \in \mathbb{R}^n} G(x)=0$ (and so, $\nabla G(a)=0$). As $G$ is an SOS-convex polynomial, $H(x,y):=G(x)-G(y)-\nabla G(y)^T(x-y)$ is a sums-of-squares polynomial. Letting $y=a$, it follows that
$G(x)=H(x,a)$ is also a sums-of-squares polynomial, that is,  \begin{equation}\label{eq:pp1}
h_0^0+\sum_{j=1}^m \tilde{y}_j^0 h_j^0+\sum_{i=1}^s \tilde{\lambda}_i\big(h_0^i+\sum_{j=1}^m \tilde{y}_j^i h_j^i\big)-r \in \Sigma^2_d[x].
\end{equation}
On the other hand, for each $i=0,...,s,$ since  $(\tilde{y}_1^i,\ldots,\tilde{y}_m^i) \in \Omega_i,$  there exists $ \tilde z^i=(\tilde z^i_1,...,\tilde z^i_{p_i})\in \R^{p_i}$ such that
\begin{equation}\label{eq:pp}
A_0^i+\sum_{j=1}^m \tilde y_j^i A_j^i+ \sum_{l=1}^{p_i}  \tilde z_l^i B_l^i \succeq 0.\end{equation}
Now, let  $\lambda^0_j:=\tilde y_j^0$, $j=1,\ldots,m$, $z^0_l:=\tilde z^0_l,$ $l=1,...,p_0,$   $\lambda_0^i:=\tilde{\lambda}_i$ and $\lambda_j^i:=\tilde{\lambda}_i \tilde{y}_j^i$  and $z_l^i:=\tilde{\lambda}_i \tilde{z}_l^i$  for each $j=1,\ldots,m,$  $i=1,\ldots,s,$ $l=1,...,p_i$.
From \eqref{eq:pp1}, \eqref{eq:pp} and $\tilde{\lambda}_i \ge 0$,
we see that $$ h_0^0+\sum_{j=1}^m \lambda^0_j  h_j^0+\sum_{i=1}^s \left(\lambda_0^i h_0^i + \sum_{j=1}^m \lambda_j^i h_j^i\right) -r \in \Sigma^2_d[x],$$\
$$A_0^0+\sum_{j=1}^m\lambda_j^0 A_j^0+\sum_{l=1}^{p_0} z_l^0 B_l^0 \succeq 0,$$
and $$\lambda_0^i A_0^i+\sum_{j=1}^m\lambda_j^i A_j^i+\sum_{l=1}^{p_i} z_l^i B_l^i \succeq 0, \, i=1,\ldots,s.$$
This says that  $  \lambda_0^i \ge 0, (\lambda_1^i,....,\lambda_m^i) \in \mathbb{R}^m,  z_l^i \in \mathbb{R}, r \in \mathbb{R}$ is feasible for $(SDP).$
Thus,
$${\rm val}(P) =r\le {\rm val}(SDP),$$ and hence ${\rm val}(P) = {\rm val}(SDP)$. The proof is complete.
\end{proof}

\begin{remark}{\bf (Special Cases: Min-max programs involving SOS-convex polynomials)}
In the case  where the objective function $f_0$ can be expressed as a finite maximum of SOS-convex polynomials and the constraint functions $f_i$, $i=1,\ldots,s$, are SOS-convex polynomials, Theorem \ref{th:exact_SDP} has established in \cite[Theorem 3.1]{minmax} for min-max programs.
\end{remark}
In the preceding theorem, we see that the optimal value of a SOS-convex semi-algebraic optimization problem (P) can be found by solving a single semi-definite programming problem, that is, its SDP relaxation problem (SDP). Next, we examine the important question that: how to recover an optimal solution of (P) from its SDP relaxation problem?

For a given $z=(z_{\alpha}) \in \mathbb{R}^{s(r,n)}$,  we define a linear function $L_z: \mathbb{R}_r[x_1,\ldots,x_n] \rightarrow \mathbb{R}$ by
\begin{equation}\label{Reisz}
L_z(u)=\sum_{\alpha=1}^{s(r,n)} u_{\alpha} z_{\alpha} \mbox{ with } u(x)=\sum_{\alpha=1}^{s(r,n)} u_\alpha x^{(r)}_{\alpha}.
\end{equation}
For each $\alpha=1,\ldots,s(2r,n)$, define $M_{\alpha}$ to be the $(s(r,n) \times s(r,n))$ symmetric matrix such that
\[
{\rm Tr} (M_{\alpha} W)= \sum_{\substack{1 \le \beta,\gamma \le s(r,n) \\
i(\beta)+i(\gamma)=i(\alpha)}} W_{\beta,\gamma} \mbox{ for all } W \in S^{s(r,n)}.
\]
Then, for $z=(z_{\alpha}) \in \mathbb{R}^{s(2r,n)}$, the moment matrix with respect to the sequence $z=(z_{\alpha})$ with degree $r$ is denoted by ${\bf M}_r(z)$, and is defined by
\[
{\bf M}_r(z)=\sum_{1 \le \alpha \le s(2r,n)} z_{\alpha}M_{\alpha}.
\]

As a simple illustration, let $r=4$ and $n=1$, for  $z=(z_1,\ldots,z_5)^T \in \mathbb{R}^{s(4,1)}=\mathbb{R}^5$,
\[
L_z(u)= \sum_{i=1}^{5} \alpha_iz_i, \mbox{ for all } u(x)=\alpha_1+\alpha_2 x+\alpha_3 x^2+\alpha_4x^3+ \alpha_5 x^4.
\]
Moreover, for $r=1$, $n=2$ and $z \in \mathbb{R}^{s(2,2)}=\mathbb{R}^{6}$
\[
{\bf M}_1(z)= \left(\begin{array}{ccc}
                    z_1 & z_2 & z_3 \\
                    z_2 & z_4 & z_5 \\
                    z_3 & z_5 & z_6
                    \end{array}
\right).
\]
Recall  that $(SDP)$ can be equivalently rewritten as a semi-definite programming problem.
{\small \begin{eqnarray*}
& & \sup_{\substack{\lambda_0^i \ge 0,   (\lambda_1^i,...,\lambda_m^i) \in \mathbb{R}^m\\
 z_l^i \in \mathbb{R}, \mu \in \mathbb{R},
W \in S^{s(d/2,n)}}} \{\mu:   (h_0^0)_1+\sum_{j=1}^m  \lambda_j^0 (h_j^0)_1+\sum\limits_{i=1}^{s}\left(\lambda_0^i (h_0^i)_1 + \sum_{j=1}^m \lambda^i_j (h_j^i)_1\right)-\mu =W_{1,1}, \nonumber \\
& & \ \ \ \ \ \ \ \ \ \ \ \ \ \ \ \ \ \ \ \ \  (h_0^0)_\alpha+\sum_{j=1}^m \lambda_j^0 (h_j^0)_\alpha+\sum\limits_{i=1}^{s}\left(\lambda_0^i (h_0^i)_\alpha + \sum_{j=1}^m \lambda_j^i (h_j^i)_\alpha\right)=\sum_{\substack{ 1\le \beta,\gamma \le s(d/2,n) \\i(\beta)+i(\gamma)=i(\alpha)}}W_{\beta,\gamma} \, , \ 2  \le \alpha \le s(d,n)\nonumber \\
%f+\sum_{i=1}^m \left(\lambda_i^0 g_i^0 + \sum_{j=1}^s \lambda_i^j g_i^j\right) -\mu \in \Sigma^2_d, \\
& & \ \ \ \ \ \ \ \ \ \ \ \ \ \ \ \ \ \ \ \ \  W \succeq 0,  \quad   A_0^0+\sum_{j=1}^m\lambda_j^0 A_j^0+\sum_{l=1}^{p_0} z_l^0 B_l^0 \succeq 0,\\
& & \ \ \ \ \ \ \ \ \ \ \ \ \ \ \ \ \ \ \ \ \ \lambda_0^i A_0^i+\sum_{j=1}^m\lambda_j^i A_j^i+\sum_{l=1}^{p_i} z_l^i B_l^i \succeq 0, \, i= 1,\ldots,s\}. \nonumber
\end{eqnarray*}}
The Lagrangian dual  of the above semi-definite programming reformulation of $(SDP)$ can be stated  as follows:
$${\small \begin{array}{rl}
 &\inf\limits_{\substack{y=(y_{\alpha}) \in \mathbb{R}^{s(d,n)} \\ Z_i \succeq 0}}\,  \sup\limits_{\substack{\lambda_0^i \ge 0,   (\lambda_1^i,...,\lambda_m^i) \in \mathbb{R}^m\\
 z_l^i \in \mathbb{R}, \mu \in \mathbb{R},
W\succeq0}}\Big\{\mu +  y_1\left( (h_0^0)_1+\sum\limits_{j=1}^m   \lambda_j^0 (h_j^0)_1+\sum\limits_{i=1}^{s}\left(\lambda_0^i (h_0^i)_1 + \sum\limits_{j=1}^m \lambda^i_j (h_j^i)_1\right)-\mu -W_{1,1}\right) \\
 & +\sum\limits_{2 \le \alpha \le s(d,n)} y_{\alpha}\left((h_0^0)_\alpha+\sum\limits_{j=1}^m   \lambda_j^0  (h_j^0)_\alpha+\sum\limits_{i=1}^{s}\left(\lambda_0^i (h_0^i)_\alpha + \sum\limits_{j=1}^m \lambda_j^i (h_j^i)_\alpha\right)-\sum\limits_{\substack{ 1\le \beta,\gamma \le s(d/2,n) \\i(\beta)+i(\gamma)=i(\alpha)}}W_{\beta,\gamma}\right) \\
&+{\rm Tr}\big(Z_0 (A_0^0+\sum\limits_{j=1}^m\lambda_j^0 A_j^0+\sum\limits_{l=1}^{p_0} z_l^0 B_l^0)\Big) +\sum\limits_{i=1}^s{\rm Tr}\big(Z_i (\lambda_0^i A_0^i+\sum\limits_{j=1}^m\lambda_j^i A_j^i+\sum\limits_{l=1}^{p_i} z_l^i B_l^i\Big)\Big\},
\end{array}}$$
which   can be further simplified as
\begin{eqnarray*}
 (SDP^*) &  \displaystyle \inf_{y=(y_{\alpha}) \in \mathbb{R}^{s(d,n)}, Z_i \succeq 0} & \sum_{1 \le \alpha \le s(d,n)} (h^0_0)_{\alpha} y_{\alpha}+{\rm Tr}\big(Z_0 A_0^0\big) \\
 & \mbox{ s.t. } &   \sum_{1 \le \alpha \le s(d,n)} (h^i_0)_{\alpha} y_{\alpha}+{\rm Tr} \big( Z_i A^i_0\big) \le 0, \, i=  1,\ldots,s, \\
 & &  \sum_{1 \le \alpha \le s(d,n)} (h_j^i)_{\alpha} y_{\alpha}+{\rm Tr} \big( Z_i A_j^i\big) = 0, \, i=0, 1,\ldots,s, j=1,\ldots,m, \\
  & &{\rm Tr} \big( Z_i B_l^i\big) = 0,\, i=0, 1,...,s,\, l=1,...,p_i,\\
 & & {\bf M}_{\frac{d}{2}}(y)=\sum_{1 \le \alpha \le s(d,n)} y_{\alpha} M_{\alpha} \succeq 0,\\
 & & y_1=1.
\end{eqnarray*}
We note that the problem $(SDP^*)$  is also a semi-definite programming problem, and hence can be efficiently solved as well.

Next, we recall the following  Jensen's inequality for SOS-convex polynomial (cf \cite{Lasserre}) which will play an important role in our later analysis.
\begin{lemma}{\bf (Jensen's inequality for SOS-convex polynomial \cite[Theorem 5.13]{Lasserre})} \label{lemma:Jensen}
Let $f$ be an SOS-convex polynomial on $\mathbb{R}^n$ with degree $2r$. Let $y \in \mathbb{R}^{s(2r,n)}$ with $y_1=1$ and ${\bf M}_{r}(y) \succeq 0$. Then, we have
$$L_{y}(f) \ge f(L_{y}(X_1),\ldots,L_{y}(X_n)),$$
where $L_y$  is given as in (\ref{Reisz}) and $X_i$ denotes the polynomial which maps a vector in $\R^n$
to its $i$th coordinate.
\end{lemma}

The next main result of this section is the following theorem, providing the way to recover a solution to problem (P) from a solution to its SDP relaxation.

\begin{theorem}{\bf (Recovery of the  solution)}\label{th:3.2}
%For problem $(P)$  and $(LD)$, let $d$ be the smallest even number such that $d \ge \max\{{\rm deg}h_i^{j}\}.$  %Suppose that the assumptions of  Theorem \ref{th:exact_SDP} hold.
For problem $(P)$, suppose that the following strict feasibility conditions hold:
\begin{itemize}
\item[{\rm (i)}] there exists $\bar x \in \mathbb{R}^n$ such that $f_i(\bar x)<0 $  for all $i=1,...,s$;
\item[{\rm (ii)}] for each $i=0,1,\ldots,s$, there exist $\bar y^i \in \mathbb{R}^m$ and $\bar z^i \in \mathbb{R}^{p_i}$ such that  $A_0^i+\sum_{j=1}^m \bar y_j^i A_j^i+ \sum_{l=1}^{p_i} \bar z_l^i B_l^i \succ 0$.
%\[
%1
%\]
\end{itemize}
%
%and  are satisfied for $(P)$, that is,  .
%and
%there exist $\hat y=(\hat y_{\alpha}) \in \mathbb{R}^{s(d,n)}$ and $\hat Z_i \succ 0$ such that
%\[
%\left\{\begin{array}{cl}
%& \displaystyle \sum_{1 \le \alpha \le s(d,n)} (h_0^i)_{\alpha} \hat y_{\alpha}+{\rm Tr} \big( \hat Z_i A_0^i\big) \le 0, \, i= 0, 1,\ldots,s,  \\
% &  \displaystyle  \sum_{1 \le \alpha \le s(d,n)} (h_i^j)_{\alpha} \hat y_{\alpha}+{\rm Tr} \big(\hat Z_i  A_j^i\big) = 0, \, i=0, 1,\ldots,m, j=1,\ldots,s, \\
% & {\rm Tr} \big( \hat Z_i B_l^i\big) = 0,\, i=0, 1,...,s,\, l=1,...,p_i,\\
% & {\bf M}_{\frac{d}{2}}(\hat y) \succ 0,\\
% &  \hat y_1=1.
% \end{array}
% \right.
%\]
Let $(y^*,Z_0^*,Z_1^*,...,Z_s^*)$ be an optimal solution for $(SDP^*)$  and let $x^*:=(L_{y^*}(X_1),\ldots,L_{y^*}(X_n))^T \in \mathbb{R}^n$ where $X_i$ denotes the polynomial which maps a vector $x\in \R^n$
to its $i$th coordinate. Then, $x^*$ is an optimal solution for $(P)$.
\end{theorem}
\begin{proof}  From condition {\rm (i)},  the exact SDP relaxation result (Theorem \ref{th:exact_SDP}) gives us that ${\rm val}(P)={\rm val}(SDP)$. Note that $(SDP)$ and $(SDP^*)$ are dual problems to each other. The usual weak duality for semi-definite programming implies that
${\rm val}(SDP^*) \ge {\rm val}(SDP)={\rm val}(P)$. Next, we establish  that ${\rm val}(SDP^*)={\rm val}(P)$, where ${\rm val}(SDP^*)$ is the optimal value of problem $(SDP^*)$.  To see this, let $x$ be a feasible point of $(P)$ and let $r=f_0(x)$.
Then
$$f_0(x)=\sup_{(y_1^0,\ldots,y_m^0) \in \Omega_0} \{h_0^0(x)+\sum_{j=1}^m y_j^0 h_j^0(x)\}=r$$
 and 
 $$f_i(x)=\sup_{(y_1^i,\ldots,y_m^i) \in \Omega_i} \{h_0^i(x)+\sum_{j=1}^m y_j^i h_j^i(x)\} \le 0, \, i=1,\ldots,s,$$ where $\Omega_i$, $i=0,1,\ldots,s$ are compact sets given by
$\Omega_i=\big\{(y_1^i,\ldots,y_m^i) \in \mathbb{R}^m: \exists  z^i=(z^i_1,...,z^i_{p_i})  \in \mathbb{R}^{p_i} \mbox{ s.t. } A_0^i+\sum_{j=1}^m y_j^i A_j^i+ \sum_{l=1}^{p_i} z_l^i B_l^i \succeq 0\big\}$. This shows that
\[
(y^0,z^0) \in \mathbb{R}^m \times \mathbb{R}^{p_0}, \ A_0^0+\sum_{j=1}^m y_j^0 A_j^0+ \sum_{l=1}^{p_0} z_l^0 B_l^0 \succeq 0 \, \Rightarrow \, h_0^0(x)+\sum_{j=1}^m y_j^0 h_j^0(x) \le r,
\]
and
\[
(y^i,z^i) \in \mathbb{R}^m \times \mathbb{R}^{p_i}, \ A_0^i+\sum_{j=1}^m y_j^i A_j^i+ \sum_{l=1}^{p_i} z_l^i B_l^i \succeq 0 \, \Rightarrow \, h_0^i(x)+\sum_{j=1}^m y_j^i h_j^i(x) \le 0, \ i=1,\ldots,s.
\]
It then follows from condition {\rm (ii)} and the strong duality theorem for semi-definite programming  that there exist $Z_i \succeq 0$, $i=0,1,\ldots,s$ such that
\[
\left\{ \begin{array}{cc}
& h_0^0 (x)+{\rm Tr} \big( Z_0 A_0^0\big) \le r,  \\
& h^i_0(x)+{\rm Tr}\big(Z_iA^i_0\big)\le 0,\, i=1,...,s,\\
 &   h_j^i(x)+{\rm Tr} \big( Z_i A_j^i\big) = 0, \, i= 0, 1,\ldots,s,\,  j=1,\ldots,m, \\
  & {\rm Tr} \big( Z_i B_l^i\big) = 0,\, i=0, 1,...,s,\, l=1,...,p_i,
\end{array} \right.
\]
Let $x^{(d)}=(1,x_1,x_2,\ldots,x_n,x_1^2,x_1x_2,\ldots,x_2^2,\ldots,x_n^2,\ldots,x_1^{d},\ldots,x_n^d)^T$. Then, $(x^{(d)},Z_0,Z_1,\ldots,Z_s)$  is feasible for $(SDP^*)$  and
\[
f_0(x)=r \ge h_0^0 (x)+{\rm Tr} \big( Z_0 A_0^0\big)= \sum\limits_{1 \le \alpha \le s(d,n)} (h_0^0)_{\alpha} x^{(d)}_{\alpha}+{\rm Tr} \big( Z_0 A_0^0\big).
\]
This shows that ${\rm val}(P) \ge {\rm val}(SDP^*)$, and hence ${\rm val}(P)={\rm val}(SDP^*)$.

Now, let $(y^*,Z_0^*,Z_1^*,...,Z_s^*)$ be an optimal solution for $(SDP^*)$. Then, $Z_i^* \succeq 0$, $i=0,1,...,s$, and
$$ \begin{array}{rl} & \sum\limits_{1 \le \alpha \le s(d,n)} (h_0^i)_{\alpha} y_{\alpha}^*+{\rm Tr} \big( Z_i^* A_0^i\big) \le 0, \, i=1,\ldots,s, \\
 &  \sum\limits_{1 \le \alpha \le s(d,n)} (h^i_j)_{\alpha} y^*_{\alpha}+{\rm Tr} \big( Z_i^* A_j^i\big) = 0, \, i=0, 1,\ldots,s, j=1,\ldots,m, \\
  &{\rm Tr} \big( Z_i^* B_l^i\big) = 0,\, i=0, 1,...,s,\, l=1,...,p_i,\\
 &  {\bf M}_{\frac{d}{2}}(y^*)=\sum\limits_{1 \le \alpha \le s(d,n)} y^*_{\alpha} M_{\alpha} \succeq 0,\\
 &  y_1^*=1.\end{array}$$
 Note that for each $(y^0_1,...,y^0_m)\in \Omega_0,$ one can find  $z^0=(z^0_1,...,z^0_{p_0})  \in \mathbb{R}^{p_0}$ such that
 $$A^0_0+\sum\limits_{j=1}^my^0_j A_j^0+\sum\limits_{l=1}^{p_0}z^0_lB_l^0\succeq0.$$
So,  for each $(y^0_1,...,y^0_m)\in \Omega_0,$ it holds that
 \begin{equation}\label{eq1}\begin{array}{rl}  \sum\limits_{1 \le \alpha \le s(d,n)} (h_0^0)_{\alpha} y_{\alpha}^*+{\rm Tr}\big( Z_0^* A_0^0\big)&\geq \sum\limits_{1 \le \alpha \le s(d,n)} (h_0^0)_{\alpha} y_{\alpha}^* -{\rm Tr} \big( Z_0^*(\sum\limits_{j=1}^my^0_j A_j^0+\sum\limits_{l=1}^{p_0}z^0_lB_l^0)\big)\\
 &=\sum\limits_{1 \le \alpha \le s(d,n)} (h_0^0)_{\alpha} y_{\alpha}^*-\sum\limits_{j=1}^my^0_j{\rm Tr} \big( Z_0^*A_j^0\big)\\
 &=\sum\limits_{1 \le \alpha \le s(d,n)} (h_0^0)_{\alpha} y_{\alpha}^*+\sum\limits_{j=1}^my^0_j  \sum\limits_{1 \le \alpha \le s(d,n)}(h^0_j)_{\alpha} y^*_{\alpha}\\
 & =L_{y^*}\big(h_0^0+\sum\limits_{j=1}^my^0_jh^0_j\big).
 \end{array}\end{equation}
 Since $h_0^0+\sum\limits_{j=1}^my^0_jh^0_j$ is SOS-convex,  ${\bf M}_{\frac{d}{2}}(y^*)\geq 0,$  and  $y_1^*=1,$  by Lemma \ref{lemma:Jensen}, we have
 \begin{equation}\label{eq2}L_{y^*}\big(h_0^0+\sum\limits_{j=1}^my^0_jh^0_j\big)\geq \big(h_0^0+\sum\limits_{j=1}^my^0_jh^0_j\big)(L_{y^*}(X_1),\ldots,L_{y^*}(X_n))=h_0^0(x^*)+\sum\limits_{j=1}^my^0_jh^0_j(x^*)\end{equation}
 for every $(y^0_1,...,y^0_m)\in \Omega_0.$  Taking supremum over all $(y^0_1,...,y^0_m)\in \Omega_0$ in  \eqref{eq1} and using \eqref{eq2}, it follows that
 $$f_0(x^*)\leq \sum\limits_{1 \le \alpha \le s(d,n)} (h_0^0)_{\alpha} y_{\alpha}^*+{\rm Tr}\big( Z_0^* A_0^0\big).$$
 Taking into account  that  $(y^*,Z_0^*,Z_1^*,...,Z_s^*)$ is an optimal solution for $(SDP^*),$ we get
$$ {\rm val}(SDP^*)=\sum_{1 \le \alpha \le s(d,n)} (h_0^0)_{\alpha} y_{\alpha}^*+{\rm Tr} \big( Z_0^* A_0^0\big)  \ge  f_0(x^*).$$
 We claim that  $x^*$ is feasible for (P).   Granting this, we have    $$ {\rm val}(SDP^*)\geq  f_0(x^*) \ge {\rm val}(P)={\rm val}(SDP^*).$$
This forces that $f_0(x^*)={\rm val}(P)$, and so, $x^*$ is an optimal solution for (P).

We now verify our claim. Take any $i=1,...,s$ and $(y^i_1,...,y^i_m)\in \Omega_i.$  Then one can find  $z^i=(z^i_1,...,z^i_{p_i})  \in \mathbb{R}^{p_i}$ such that
 $$A^i_0+\sum\limits_{j=1}^my^i_j A_j^i+\sum\limits_{l=1}^{p_i}z^i_lB_l^i\succeq0.$$
 Arguing as before,   we arrive at
$$ f_i(x^*)\leq \sum\limits_{1 \le \alpha \le s(d,n)} (h_0^i)_{\alpha} y_{\alpha}^*+{\rm Tr} \big( Z_i^* A_0^i\big)\leq 0.$$
 This shows that $x^*$ is feasible for (P).  So, the conclusion follows.
\end{proof}

Finally, we illustrate how to find the optimal value and an optimal solution for an SOS-convex semi-algebraic program by solving a single
semi-definite programming problem.
\begin{example}{\bf (Illustrative example)}{\rm
 Consider the following simple 2-dimensional nonsmooth convex optimization problem:
   \begin{eqnarray*}
(EP) & \min & x_1^4-x_2 \\
& \mbox{\rm s.t. } & x_1^2+ x_2^2+ 2\|(x_1,x_2)\| -1\le 0. % \, \forall \, (u_1^1,u_1^2) \in \mathcal{U},
\end{eqnarray*}
Let
\[
\Omega_1=\{(y_1^1,y_2^1): (y_1^1)^2+(y^1_2)^2 \le 1\}= \{(y_1^1,y^1_2): \left(\begin{array}{ccc}
                                  1 & 0 & 0 \\
                                  0 & 1 & 0 \\
                                  0 & 0 & 1
                                  \end{array}
 \right)+ y_1^1 \left(\begin{array}{ccc}
                                  0 & 0 & 1 \\
                                  0 & 0 & 0 \\
                                  1 & 0 & 0
                                  \end{array}
 \right)+y^1_2\left(\begin{array}{ccc}
                                  0 & 0 & 0 \\
                                  0 & 0 & 1 \\
                                  0 & 1 & 0
                                  \end{array}
 \right) \succeq 0\}.
\]
 Let $h^1_0(x)=x_1^2+ x_2^2 -1$ and $h^1_j(x)=2x_j$, $j=1,2$. 
We first observe that,   for each $(y^1_1,y^1_2)\in \Omega_1,$  the function  $h^1_0+\sum_{j=1}^2 y^1_j h^1_j$  is  an SOS-convex polynomial. Denote 
 $f_0(x)=x_1^4-x_2$ and $f_1(x)= x_1^2+ x_2^2+ 2\|(x_1,x_2)\| -1$. 
Then, $f_1(x)=\sup_{(y^1_1,y^1_2)\in \Omega_1}\{h^1_0(x)+y^1_1h^1_1(x)+y^1_2h^1_2(x)\}$, and so, $f_1$ is an SOS-convex semi-algebraic function. Obviously,  $f_0$ is an SOS-convex polynomial and thus is also an  SOS-convex semi-algebraic function. This shows  that  (EP) is an SOS-convex semi-algebraic program.

Let $x_0=(0,0)$. It can be verified that $f_1(x_0)=-1<0.$  %, $f$ is SOS-convex and $g_1(\cdot,y_1)$ is SOS-convex for all $y_1 \in \Omega_1$.
 Thus, Theorem \ref{th:exact_SDP} implies that ${\rm val}(EP)={\rm val}(ESDP)$ where (ESDP) is given by
\begin{eqnarray*}
(ESDP) & & \sup_{\lambda^1_0 \ge 0, \lambda^1_j \in \mathbb{R}, \mu \in \mathbb{R}} \{\mu: f_0+\left(\lambda^1_0 h^1_0 + \sum_{j=1}^2\lambda^1_j h^1_j\right) -\mu \in \Sigma^2_4[x], \\
& & \ \ \ \ \ \ \ \ \ \ \ \ \ \ \ \ \ \ \ \ \  \lambda^1_0 \left(\begin{array}{ccc}
                                  1 & 0 & 0 \\
                                  0 & 1 & 0 \\
                                  0 & 0 & 1
                                  \end{array}
 \right)+ \lambda^1_1 \left(\begin{array}{ccc}
                                  0 & 0 & 1 \\
                                  0 & 0 & 0 \\
                                  1 & 0 & 0
                                  \end{array}
 \right)+\lambda^1_2 \left(\begin{array}{ccc}
                                  0 & 0 & 0 \\
                                  0 & 0 & 1 \\
                                  0 & 1 & 0
                                  \end{array}
 \right) \succeq 0\}.
\end{eqnarray*}
Note that \begin{eqnarray*}
 & & f_0+\left(\lambda^1_0 h_1^0 + \sum_{j=1}^2\lambda^1_j h^1_j\right) -\mu \in \Sigma^2_4[x] \\
 &\Leftrightarrow &
x_1^4-x_2+\lambda^1_0 (x_1^2+x_2^2-1)+ 2\lambda_1^1x_1+ 2\lambda^1_2x_2-\mu \\
& & =\left(1, x_1,x_2,x_1^2,x_1x_2,x_2^2\right)\left(\begin{array}{cccccc}
W_{11} & W_{12} & W_{13} & W_{14} & W_{15} & W_{16}\\
W_{12} & W_{22} & W_{23} & W_{24}& W_{25} & W_{26}\\
W_{13} & W_{23} & W_{33} & W_{34}& W_{35} & W_{36}\\
W_{14} & W_{24} & W_{34} & W_{44}& W_{45} & W_{46}\\
W_{15} & W_{25} & W_{35} & W_{45}& W_{55} & W_{56}\\
W_{16} & W_{26} & W_{36} & W_{46}& W_{56} & W_{66}
\end{array}\right)\left(\begin{array}{c}
1 \\
x_1 \\
x_2 \\
x_1^2 \\
x_1x_2 \\
x_2^2 \end{array}\right), W=(W_{ij}) \in S_+^{6}, \\
& \Leftrightarrow & W_{11}=-\lambda^1_0-\mu, W_{12}=\lambda_1^1, 2W_{14}+W_{22}=2\lambda^1_0,  \\
& & W_{33}+2W_{16}=\lambda^1_0, 2W_{13}+W_{66}=-1+2\lambda^1_2, W_{44}=1, \\
& &  W_{23}=W_{24}=W_{34}=0, W_{i5}=0, i=1,\ldots,5, W_{i6}=0, i=1,\ldots,6, \\
& &  W=(W_{ij}) \in S_+^{6}.   \end{eqnarray*}
Thus, $(ESDP)$ can be equivalently rewritten as the following semidefinite programming problem:
\begin{eqnarray*}
\sup_{\lambda^1_0 \ge 0, \lambda^1_j \in \mathbb{R}, \mu \in \mathbb{R},W \in S^6} \{\mu&:& W_{11}=-\lambda^1_0-\mu, 2W_{12}=\lambda^1_1, 2W_{14}+W_{22}=\lambda^1_0,  \\
& & W_{11}=-\lambda^1_0-\mu, W_{12}=\lambda_1^1, 2W_{14}+W_{22}=2\lambda^1_0,  \\
& & W_{33}+2W_{16}=\lambda^1_0, 2W_{13}+W_{66}=-1+2\lambda^1_2, W_{44}=1, \\
& &  W_{23}=W_{24}=W_{34}=0, W_{i5}=0, i=1,\ldots,5, W_{i6}=0, i=1,\ldots,6, \\
& &  W=(W_{ij}) \in S_+^{6} \\
& &   \lambda^1_0 \left(\begin{array}{ccc}
                                  1 & 0 & 0 \\
                                  0 & 1 & 0 \\
                                  0 & 0 & 1
                                  \end{array}
 \right)+ \lambda_1^1 \left(\begin{array}{ccc}
                                  0 & 0 & 1 \\
                                  0 & 0 & 0 \\
                                  1 & 0 & 0
                                  \end{array}
 \right)+\lambda^1_2 \left(\begin{array}{ccc}
                                  0 & 0 & 0 \\
                                  0 & 0 & 1 \\
                                  0 & 1 & 0
                                  \end{array}
 \right) \succeq 0\}.
\end{eqnarray*}
Solving this semi-definite programming problem using CVX \cite{CVX1,CVX2}, we obtain the optimal value ${\rm val}(RP)={\rm val}(ESDP)=-0.414214 \approx 1-\sqrt{2}$ and the dual variable
$y^*=(y_1^*,\ldots,y_{15}^*) \in \mathbb{R}^{15}=\mathbb{R}^{s(4,2)}$ with $y_1^*=1$, $y_2^*=0$ and $y_3^*=0.414214 \approx \sqrt{2}-1$.
It can be verified that the conditions in Theorem \ref{th:3.2} are satisfied. So, Theorem \ref{th:3.2} implies that $x^*=(L_{y^*}(X_1),L_{y^*}(X_2))=(y_2^*,y_3^*)=(0,\sqrt{2}-1)$ is a solution for (EP).

Indeed, the optimality of $(0,\sqrt{2}-1)$ for (EP) can be verified independently. To see this, note that for all $(x_1,x_2)$ which is feasible for (EP), one has %$x_1^2+ x_2^2+ 2u_1^1 x_1+2u_1^2 x_2 -1\le 0, \, \forall \, (u_1^1,u_1^2) \in \mathcal{U}$, and so,
$$x_1^2+x_2^2+2\|(x_1,x_2)\|-1 \le 0.$$
In particular,
\[
|x_2|^2+2|x_2|-1=x_2^2+2|x_2|-1 \le 0,
\]
which implies that $|x_2| \le \sqrt{2}-1$. Thus, for all feasible point $(x_1,x_2)$ for (EP), $x_1^4-x_2 \ge -x_2 \ge -|x_2| \ge 1-\sqrt{2}$, and so, ${\rm val}(EP) \ge 1-\sqrt{2}$. On the other hand, direct verification shows that $(0,\sqrt{2}-1)$ is feasible for (EP) with the object value $1-\sqrt{2}$. So, ${\rm val}(EP)=1-\sqrt{2}$ and $(0,\sqrt{2}-1)$ is a solution of the problem (EP).}
\end{example}

%\begin{verbatim}
%cvx_begin sdp
%   variables lambda0 lambda1 lambda2 mu1 a b
%   dual variables Q1 Q2
%   A=[-lambda0-mu1,       (2*lambda1)/2,   (-1+2*lambda2)/2,       b, 0, 0;
%      (2*lambda1)/2,            a,                  0,           0, 0, 0;
%   (-1+2*lambda2)/2,            0,          lambda0,           0, 0, 0;
%            b,                0,                  0,           1, 0, 0;
%            0,                0,                  0,           0, 0, 0;
%            0,                0,                  0,           0, 0, 0;]
%        2*b+a==lambda0;
%B=[lambda0 0 lambda1;
%  0 lambda0 lambda2;
%  lambda1 lambda2 lambda0];
%   maximize(mu1)
%   Q1: A==semidefinite(6);
%   Q2: B==semidefinite(3);
%cvx_end
%
%Q=full(Q1);
%y=[]
%for i=1:6
%for j=i:6
%fprintf('Q(%1d,%1d) is %4.4f \n', i,j,Q(i,j))
%y=[y Q(i,j)];
%end
%end
%fprintf('y=%4.4f \n',y)
%\end{verbatim}

%
%\section{Application: Robust SOS-convex polynomial optimization}

\section{Applications to robust optimization}
In this section, we briefly outline how our results can be applied to the area of robust optimization \cite{BEN09} (for some recent development see \cite{Goberna,Gold,jl,jl1}). Consider the following robust SOS-convex optimization problem \begin{eqnarray*}
(RP) & \min & f(x) \\
& \mbox{ subject to } & g_i^{(0)}(x)+\sum\limits_{j=1}^{t_i}u_i^{(j)}g_i^{(j)}(x)+\sum\limits_{j=t_i+1}^su_i^{(j)}g_i^{(j)}(x) \le 0, \,\ \forall u_i\in \mathcal{U}_i,\,  \ i=1,\ldots,s,
\end{eqnarray*}
where $f,$ $g_i^{(j)}$, $i=1,\ldots,m$,  $j=0, 1,...,t_i,$ are  SOS-convex  polynomials,  $g_i^{(j)}$, $i=1,\ldots,m$,  $j=t_i+1,...,s,$ are affine functions, and $u_i$ are uncertain parameters and belong to  uncertainty sets $\mathcal{U}_i$, $i=1,\ldots,s$.

In the case where $\mathcal{U}_i$ is the so-called
restricted ellipsoidal uncertainty set given by
\begin{eqnarray*}
\mathcal{U}_i^e=\{(u_i^1,\ldots,u_i^{t_i},u_i^{t_i+1},\ldots,u_i^s)&:& \|(u_i^1,\ldots,u_i^{t_i})\|\le 1, \ u_i^j \ge 0, j=1,\ldots,t_i\} \\
& & \|(u_i^{t_i+1},\ldots,u_i^{s})\| \le 1\},
\end{eqnarray*}
this robust optimization problem was first examined in \cite{Gold} in the special case of robust convex quadratic optimization problems, and then subsequently in \cite{JLV15} for general robust SOS-polynomial optimization problems. In particular, \cite{JLV15} showed that the optimal value of (RP) with $\mathcal{U}_i=\mathcal{U}_i^e$ can be found by solving a related semi-definite programming problem (SDP) and raised an open question that how to found an optimal solution of (RP) from the corresponding related SDP.

As we will see, as a simple application of the result in Section 4, we can extend the exact semi-definite programming relaxations result in \cite{JLV15} to a more general setting and answer the open questions left in \cite{JLV15} on how to recover a robust solution from the semi-definite programming relaxation in this broader setting.

To do this, we first introduce the notion of restricted spectrahedron data uncertainty set which is a compact set  %admits an exact SDP relaxation.
%To do this, we define the restricted spectrahedron
%data uncertainty set as follows.
given by
\begin{eqnarray*}
\mathcal{U}^{s}_i=\{(u_i^{(1)},\ldots,u_i^{(t_i)},u_i^{(t_i+1)},\ldots,u_i^{(s)})  \in \mathbb{R}^{s}&: & A_i^0+\sum_{j=1}^s u_i^{(j)} A_i^j \succeq 0, \\
& & (u_i^{(1)},\ldots,u_i^{(t_i)}) \in \mathbb{R}^{t_i}_+, (u_i^{(t_i+1)},\ldots,u_i^{(s)}) \in \mathbb{R}^{s-t_i} \}.
\end{eqnarray*}
It is not hard to see that the restricted ellipsoidal uncertainty set is a special case of the restricted spectrahedron data uncertainty set as the norm constraint can be expressed as a linear matrix inequality.

Let $f_0(x)=f(x)$, $g_i(x,u_i)=g_i^{(0)}(x)+\sum\limits_{j=1}^{t_i}u_i^{(j)}g_i^{(j)}(x)+\sum\limits_{j=t_i+1}^su_i^{(j)}g_i^{(j)}(x)$ and $f_i(x)=\sup_{u_i \in \mathcal{U}_i}\{g_i(x,u_i)\}$, $i=1,\ldots,s$. From the construction of
the restricted spectrahedron data uncertainty, for
each $u_i \in \mathcal{U}_i^s$, $g_i(\cdot,u_i)$ is an SOS-convex polynomial.  Moreover, each uncertainty set $\mathcal{U}_i^s$ can be written as \begin{eqnarray*}
\mathcal{U}^{s}_i=\{(u_i^{(1)},\ldots,u_i^{(t_i)},u_i^{(t_i+1)},\ldots,u_i^{(s)})  \in \mathbb{R}^{s}&: & \widetilde{A}_i^0+\sum_{j=1}^s u_i^{(j)} \widetilde{A}_i^j \succeq 0\},
\end{eqnarray*}
where
\begin{equation}\label{A_i}
\widetilde{A}_i^0=\left(\begin{array}{cc}
       0_{t_i \times t_i} & 0\\
       0 & A_i^0
                    \end{array}
 \right),  \widetilde{A}_i^j=\left(\begin{array}{cc}
       {\rm diag} \,  e_j & 0\\
       0 & A_i^j
                    \end{array}
 \right), j=1,\ldots,t_i, \mbox{ and } \widetilde{A}_i^j=\left(\begin{array}{cc}
       0_{t_i \times t_i} & 0\\
       0 & A_i^j
                    \end{array}
 \right), j=t_i+1,\ldots,s.
\end{equation}
Here, $e_j \in \mathbb{R}^{n}$ denotes the vector whose $j$th element equals to one and $0$ otherwise.
Therefore, we see that the robust convex problem (RP) under the restricted spectrahedron data uncertainty (that is, $\mathcal{U}_i=\mathcal{U}_i^s$) can be regarded as a special SOS-convex semi-algebraic program. Therefore, Theorem \ref{th:exact_SDP} and Theorem \ref{th:3.2} can be applied directly to obtain the desired exact SDP relaxation result and the exact solution recovery property. For brevity, we omit the details here.   This extends the exact semi-definite programming relaxations result in \cite{JLV15} to a more general setting and answer the open questions left in \cite{JLV15} on how to recover a robust solution from the semi-definite programming relaxation in this broader setting.

\end{document}